\documentclass[a4paper,oneside]{amsart}
\usepackage{amssymb}
\usepackage{xypic}
\usepackage[pagebackref]{hyperref}

\DeclareMathOperator{\Ker}{Ker}
\DeclareMathOperator{\Img}{Im}

\def\FL{\mbox{\it FL\/}}
\def\GPTW{\mbox{\it GPTW\/}}
\def\RC{\mbox{\it RC\/}}
\def\RA{\mbox{\it RA\/}}

\def\id{{\mathrm{id}}}
\def\CC{\mathbb{C}}
\def\ZZ{\mathbb{Z}}

\def\Zg{\ZZ_{\geq 0}}
\def\Zm{\Zg^m}
\def\og{\overline{g}}
\def\K{\mathcal{K}}
\def\L{\mathcal{L}}
\def\ZK{\mathcal{Z}_\K}
\def\DJ{\mathrm{DJ}}
\def\k{\mathbf{k}}
\def\H{{\widetilde{H}}}
\def\pp{\begin{picture}(5,5)
\put(0,2){\circle*{3}}
\put(5,2){\circle*{3}}
\end{picture}}
\def\ppp{
\begin{picture}(10,5)
\put(0,2){\circle*{3}}
\put(5,2){\circle*{3}}
\put(10,2){\circle*{3}}
\end{picture}
}
\def\pepp{
\begin{picture}(10,5)
\put(0,2){\circle*{3}}
\put(0,2){\line(1,0){5}}
\put(5,2){\circle*{3}}
\put(10,2){\circle*{3}}
\end{picture}
}

\newtheorem{theorem}{Theorem}[section]
\newtheorem{lemma}[theorem]{Lemma}
\newtheorem{conjecture}[theorem]{Conjecture}
\newtheorem{problem}[theorem]{Problem}
\newtheorem{proposition}[theorem]{Proposition}
\newtheorem{corollary}[theorem]{Corollary}
\theoremstyle{definition}
\newtheorem{definition}[theorem]{Definition}
\newtheorem{remark}[theorem]{Remark}
\newtheorem{example}[theorem]{Example}

\numberwithin{equation}{section}

\hypersetup{
    colorlinks=true,
    linkcolor=blue,
    citecolor=blue,
    urlcolor=blue,
    pdftitle={The commutator subalgebra of the Lie algebra associated with a RACG},
    pdfauthor={F. Vylegzhanin and Ya. Veryovkin},
}

\title[Commutator subalgebra of the Lie algebra associated with a RACG]{The commutator subalgebra of the Lie algebra associated with a right-angled Coxeter group}


\author{Fedor Vylegzhanin}
\address{
\parbox{\linewidth}
{National Research University Higher School of Economics, Moscow, Russia;\\
Steklov Mathematical Institute of Russian Academy of Sciences, Moscow, Russia}}
\email{vylegf@gmail.com}
\author{Yakov Veryovkin}
\address{
\parbox{\linewidth}
{National Research University Higher School of Economics, Moscow, Russia
}}
\email{verevkin{\_}j.a@mail.ru}

\thanks{This work was supported by the Russian Science Foundation under grant no.~24-71-00059, \href{https://rscf.ru/en/project/24-71-00059/}{https://rscf.ru/en/project/24-71-00059/}.}

\begin{document}
\begin{abstract}
We study the graded Lie algebra $L(RC_{\mathcal K})$ associated with the lower central series of a right-angled Coxeter group. We construct a surjective homomorphism from the polynomial ring over an explicit Lie algebra $N_{\mathcal K}$ to the commutator subalgebra of $L(RC_{\mathcal K})$, and conjecture that it is an isomorphism. The homomorphism is defined in terms of a new operation in Lie algebras associated with groups generated by involutions, which corresponds to the squaring and has an analogue in homotopy theory.

We show that the universal enveloping algebra $U(N_{\mathcal K})$ is isomorphic to the mod $2$ loop homology algebra of the corresponding moment-angle complex $\mathcal{Z}_{\mathcal K}$. This allows us to give a presentation of the Lie algebra $N_{\mathcal K}$ by generators and relations.
\end{abstract}
\maketitle
\section{Introduction}
Let $\Gamma$ be a simple graph on the vertex set $\{1,\dots,m\}$.
The corresponding \emph{right-angled Coxeter group} $\RC_\Gamma$ is the group generated by the elements $g_1,\ldots,g_m$ which satisfy the relations $g_i^2=1$ for all $i=1,\dots,m$ and $g_ig_j=g_jg_i$ for all edges $\{i,j\}\in\Gamma$. Right-angled Coxeter groups are classical objects in geometric group theory \cite{davis-book}. It is convenient to consider $\Gamma$ as the 1-skeleton of its clique complex $\K=\K(\Gamma)$ --- the largest simplicial complex with 1-skeleton $\K^1=\Gamma$. Throughout the paper, we denote $\RC_\K:=\RC_\Gamma$.

Each group $G$ has a descending filtration $G=:\gamma_1(G)\supset\dots\supset \gamma_n(G)\supset\dots$ by the normal subgroups $\gamma_{k+1}(G):=(G,\gamma_k(G))$, called the \emph{lower central series}, and the associated graded Lie algebra $L(G)=\bigoplus_{k\geq 1}L_k(G)$, $L_k(G):=\gamma_k(G)/\gamma_{k+1}(G)$, $\deg L_k=k$. We study the Lie algebra $L(\RC_\K)$ associated with a right-angled Coxeter group, and its commutator subalgebra $L'(\RC_\K)=\bigoplus_{k\geq 2}L_k(\RC_\K)$.

For the right-angled \emph{Artin} groups $\RA_\K$, defined similarly to the right-angled Coxeter groups but without introducing the relations $g_i^2 = 1$, the associated Lie algebra was computed in \cite{Duch-Krob,Papa-Suci,WaDe}. In more detail, the Lie algebra $L(\RA_\K)$ is isomorphic to the graph Lie algebra (over $\ZZ$) which corresponds to the graph $\K^1$. Similarly, by \cite[Proposition 4.2]{veryovkin} there is an epimorphism $\varphi: L_\K\to L(\RC_\K)$ of Lie algebras, where $L_\K$ is the graph Lie algebra over $\ZZ_2$, see \eqref{eqn:definition-of-LK}. This epimorphism is not an isomorphism (see~\cite[Example 4.3]{veryovkin}). So the associated Lie algebra is much more complicated for right-angled Coxeter groups than for right-angled Artin groups. In~\cite[Theorem 4.5]{veryovkin}, the groups $L_k(\RC_\K)$ were described for $k\leq 3$. For $k\geq 4$ difficulties arose, similar to those encountered in~Struik \cite{Struik1,Struik2} in the calculation of the quotient groups $\gamma_1(\RC_\K) / \gamma_k(\RC_\K)$ in some particular cases.

For any group $G$ generated by involutions, we introduce a new linear operation $h:L_k(G)\to L_{k+1}(G)$ which satisfies the identity $h([x,y])=[h(x),y]$, corresponds to the squaring in $G$ and has an analogue in homotopy theory (see Theorem \ref{t:h_properties} and Remark \ref{r:h_homotopy}).
Using this operation, in Theorem \ref{t:main_surjective_homomorphism} we define a surjective homomorphism of Lie algebras $$\psi:N_\K[t]=\bigoplus_{i\geq 0}N_\K t^i\to L'(\RC_\K)=\bigoplus_{k\geq 2}L_k(\RC_\K),$$ where $N_\K:=\langle\GPTW\rangle_{Lie}\subset L_\K$ is the Lie subalgebra generated by the explicit set of \emph{Grbi\'c--Panov--Theriault--Wu generators} \cite{gptw}. We prove that $\psi$ is injective in graded components of degree $k=2,3$ (see Proposition \ref{p:partial_confirmation}), and conjecture that $\psi$ is an isomorphism. The results of the unpublished thesis of R. Prener \cite{prener} confirm the conjecture in the case when $\K$ is a disjoint union of simplices.

We also obtain a remarkable connection of the Lie algebra $N_\K$ with the homotopy theory of \emph{moment-angle complexes} $\ZK$, which are topological spaces important in toric topology and theory of polyhedral products \cite{bu-pa15,bbc}. In Theorem \ref{t:U(N_K)} we prove an isomorphism of graded associative $\ZZ_2$-algebras
$$U(N_\K)\cong H_*(\Omega\ZK;\ZZ_2),$$
where $U(-)$ is the universal enveloping algebra of a Lie algebra and $H_*(\Omega X;\k)$ is the Pontryagin algebra of a simply connected space $X$. To prove this, in Section \ref{section:polyhedral-products} we use topological methods to obtain a partial multiplicative section of the natural homomorphism of Hopf algebras
$$T(x_1,\dots,x_m)/([x_i,x_j]=0,\{i,j\}\in\K)\to T(u_1,\dots,u_m)/(u_i^2=[u_i,u_j]=0,~\{i,j\}\in\K).$$
Using Theorem \ref{t:U(N_K)}, in Section \ref{sec:N_K} we describe the additive and multiplicative structure of $N_\K$.

\subsection{Acknowledgements}
The authors thank their teacher Taras Panov for suggesting the problem and very useful discussions, and the anonymous referee for important comments which greatly improved the exposition.

\section{Preliminaries}
\subsection{Simplicial complexes and graphs}
Let $\K$ be an abstract simplicial complex on the vertex set $[m] = \{1,2, \dots, m \}$. 
A subset $I\subset[m]$ such that $I\in\K$ is called a \emph{simplex} of $\K$. We always assume that $\K$ contains the empty set $\varnothing$ and the singletons $\{i\}$, $i = 1, \ldots, m$. For a subset~$J\subset[m]$, the simplicial complex
\[
  \K_J:=\{I\subset J\colon I\in\K\}
\]
is called the \emph{full subcomplex} of $\K$ on the vertex set $J$.

The \emph{Euler characteristic} of a complex $\K$ is the Euler characteristic of its geometric realisation. Hence $\chi(\K)=\sum_{I\in\K,I\neq\varnothing}(-1)^{|I|-1}$, so $1-\chi(\K)=\sum_{I\in\K}(-1)^{|I|}$.

A simplicial complex $\K$ is \emph{flag} if any set of pairwise connected vertices of $\K$ is a simplex. Any simple graph is the $1$-dimensional skeleton of a unique flag complex; thus a flag complex $\K$ is  determined by its $1$-skeleton $\K^1:=\{I\in\K:|I|\leq 2\}$.

A simple graph is \emph{chordal} if for any simple cycle of length $\geq 4$ there is an edge which divides it into two smaller cycles.

\begin{remark}
Many of objects considered in this paper are defined using a flag simplicial complex $\K$, but in fact are determined by its 1-skeleton, i.e. by a simple graph. We use the language of simplicial complexes to highlight the connection with the theory of polyhedral products, see \cite{pv,veryovkin,gptw}.
\end{remark}
\subsection{The Lie algebra associated with a group}
For a group $G$, the \emph{commutator} of $a, b \in G$ is the element $(a,b) = a^{-1}b^{-1}ab\in G$. For any three elements $a, b, c \in G$, the \emph{Witt--Hall identities} hold:
\begin{equation}
\begin{aligned}
&(a, bc) = (a, c) (a, b) ((a, b), c),\\
&(ab, c) = (a, c) ((a, c), b) (b, c),\\
&((a,b),c)((b,c),a)((c,a),b)=(b,a)(c,a)(c,b)^a(a,b)(a,c)^b(b,c)^a(a,c)(c,a)^b,
\end{aligned}
\end{equation}
where $a^b = b^{-1}ab$.

For any subgroups $H, W \subset G$, denote by $(H, W) \subset G$ the subgroup generated by the commutators $(h, w)$ for all $h \in H, w \in W$. In particular, the \emph{commutator subgroup} $G'$ of $G$ is $(G, G)$. Define $\gamma_1(G) := G$ and then recursively $\gamma_{k+1}(G) := (\gamma_{k}(G), G)$. The sequence of groups $\gamma_1(G), \gamma_2(G), \ldots, \gamma_k(G), \ldots$ is called the \emph{lower central series} of $G$. Then $\gamma_{k+1}(G)$ is a normal subgroup of $\gamma_k(G)$, and the quotient group $\gamma_{k}(G) / \gamma_{k+1}(G)$ is abelian. Consider their direct sum
$$
L(G) := \bigoplus_{k=1}^{+\infty} L_k (G),\quad L_k(G):=\gamma_k(G)/\gamma_{k+1}(G).
$$
We denote by $\overline{a_k}$ the class of an element $a_k \in \gamma_k(G)$ in the quotient group~$L_k (G)$. If $a_k \in \gamma_k(G), \; a_l \in \gamma_l(G)$, then $(a_k, a_l) \in \gamma_{k+l}(G)$. It follows from the Witt--Hall identities that $L(G)$ is a graded Lie algebra over $\ZZ$ (a Lie ring) with the Lie bracket defined by the formula $[\overline{a_k}, \overline{a_l}] := \overline{(a_k, a_l)}$. The Lie algebra $L(G)$ is called the \emph{Lie~algebra associated with} the group $G$ (see~\cite{Lazard, Ma-Car-Sol}).

Since the subgroup $\gamma_n(G)$ is generated by commutators of elements from $\gamma_1(G)=G$, the graded Lie algebra $L(G)$ is generated by $L_1(G)$. It follows that $\bigoplus_{k\geq 2}L_k(G)$ is the commutator subalgebra $L'(G):=[L(G),L(G)]$ of $L(G)$.

\subsection{Right-angled Coxeter groups}
Denote by $F(g_1, \ldots, g_m)$ the free group on $g_1, \ldots, g_m$. The \emph{right-angled Coxeter group} corresponding to a simplicial complex $\K$ on $m$ vertices is the group $\RC_\K$ defined by generators and relations as follows:
\[
  \RC_\K := F(g_1,\ldots,g_m)\big/ (g_i^2 = 1 \text{ for } i \in \{1, \ldots, m\}, \; \; g_ig_j=g_jg_i\text{ if
  }\{i,j\}\in\K).
\]
Clearly, the group $\RC_\K$ depends only on the graph $\K^1$, the $1$-skeleton of $\K$. Hence we can assume that $\K$ is a flag complex.

\medskip

We denote by $\ZZ_2$ the field with two elements. Denote by $\FL_{\ZZ_2}(\mu_1, \mu_2, \ldots, \mu_m)$ the free Lie algebra over $\ZZ_2$ on the generators $\mu_i$, where $\deg \mu_i = 1$. For a simplicial complex $\K$ on $m$ vertices define the \emph{graph Lie algebra} over $\ZZ_2$,
\begin{equation}
\label{eqn:definition-of-LK}    
L_\K := \FL_{\ZZ_2}(\mu_1, \mu_2, \ldots, \mu_m)/ ([\mu_i, \mu_j] = 0 \text{ for } \{i, j\} \in \K).\quad
\end{equation}
$L_\K$ also depends only on the graph $\K^1$.

\begin{proposition}[{\cite[Propositions 3.3 and 4.1]{veryovkin}}]
\label{p:liz2}
The square of any element of $\gamma_k(\RC_\K)$ is contained in $\gamma_{k+1}(\RC_\K)$. Hence
$L(\RC_\K)$ is a Lie algebra over $\ZZ_2$.\qed
\end{proposition}
\begin{proposition}[{\cite[Proposition 4.2]{veryovkin}}]
\label{p:epinmono}
There is an epimorphism of Lie algebras $\varphi : L_\K \rightarrow L(\RC_\K)$, $\varphi(\mu_i)=\og_i$.\qed
\end{proposition}
The epimorphism $\varphi$ is not an isomorphism (see \cite[Example 4.3]{veryovkin}). In this sense the right-angled Coxeter groups are different from the right-angled Artin groups, where the associated Lie algebra $L(\RA_\K)$ is isomorphic to the graph Lie algebra over $\ZZ$, see~\cite{Duch-Krob,Papa-Suci,WaDe}.

\subsection{Nested commutators}
Fix $m\geq 1$, and let $q=(q_1,\dots,q_m)$ be elements in a Lie algebra $L$ over a commutative ring with unit. (In our applications, $m$ is the number of vertices of a simplicial complex).
\begin{definition}
    An expression of the form
    $$a=[q_{i_1},[\dots[q_{i_k},q_j]\dots]],\quad k\geq 1,~i_1,\dots,i_k,j\in[m]$$
    is a  \emph{nested commutator of length $k+1$} on the elements $q_1,\dots,q_m\in L$. We say that $a$ is a nested commutator \emph{without repeats}, if $i_1,\dots,i_k,j\in[m]$ are pairwise distinct. 
\end{definition}
\begin{definition}
Let $V=(i_1,\dots,i_k)$ be a sequence of elements in $[m]=\{1,\dots,m\}$. For a subset $I\subset[k],$ $I=\{t_1<\dots<t_s\}$ and an element $x\in L$, denote
$$c_{q,V}(I,x):=[q_{i_{t_1}},[q_{i_{t_2}},\dots[q_{i_{t_s}},x]\dots]]\in L.$$
For example, $c_{q,V}(\varnothing,x)=x$ and $c_{q,V}(\{t\},x)=[q_{i_t},x]$. Denote also
$$c_q(I,x):=c_{q,(1,2,\dots,m)}(I,x)=[q_{t_1},[q_{t_2},\dots[q_{t_s},x]\dots]]$$
for $I=\{t_1<\dots<t_s\}\subset[m]$. For $V=(5,4,5,1)$ and $I=\{1,3\}$, we have $c_{q,V}(I,x)=[q_5,[q_5,x]]$ and $c_q(I,x)=[q_1,[q_3,x]]$.
\end{definition}

The next lemma is proved by induction using the Jacobi identity.
\begin{lemma}[{cf. \cite[Lemma C.3]{vylegzhanin}}]
\label{l:c(I,[x,y])}
In the notations above, for $x,y\in L$ we have
\[
c_{q,V}(I,[x,y])=\sum_{I=A\sqcup B}\Big[c_{q,V}(A,x),c_{q,V}(B,y)\Big];
\]
in particular,
\[
c_q(I,[x,y])=\sum_{I=A\sqcup B}\Big[c_q(A,x),c_q(B,y)\Big].\qed
\]
\end{lemma}

\begin{definition}[{\cite[Definition 5.2]{vylegzhanin}}]
    Let $\K$ be a flag simplicial complex on the vertex set $[m]=\{1,\dots,m\}$. For a subset $J\subset[m]$, denote by $\Theta_\K(J)\subset J$ the set of all vertices $j\in J$ which satisfy two conditions:
    \begin{itemize}
        \item $j$ is the smallest vertex in its connected component of the complex $\K_J$;
        \item the vertices $j$ and $\max(J)$ belong to different connected components of $\K_J$.
    \end{itemize}
    Given the elements $q_1,\dots,q_m$ of a Lie algebra $L$, a commutator $c_q(J\setminus j,q_j)\in L$ is called the \emph{GPTW generator} if $J\subset[m]$ and $j\in \Theta_\K(J).$
\end{definition}
I.e., GPTW generators are nested commutators without repeats of the form
$$[q_{i_1},[\dots[q_{i_k},q_j]\dots]],~i_1<\dots<i_k>j,\quad j\in\Theta_\K(\{i_1,\dots,i_k,j\}).$$
For example, $q_i\notin\GPTW$, and $[q_i,q_j]\in\GPTW$ if and only if $i>j$ and $\{i,j\}\notin\K$. The set of GPTW generators $$\GPTW:=\{c_q(J\setminus j,q_j):J\subset[m],j\in\Theta_\K(J)\}\subset L$$ consists of $\sum_{J\subset[m]}\dim\H_0(\K_J)$ elements, where $\K_J$ has $1+\dim\H_0(\K_J)$ connected components. 
These elements first appeared in the work of Grbi\'c, Panov, Theriault and Wu \cite[Theorem 4.3]{gptw} as a minimal generating set for loop homology of moment-angle complexes corresponding to flag complexes. Similar commutators in $\RC_\K$ form a minimal generating set for the group $\RC'_\K$ \cite[Theorem 4.5]{pv}.

\section{Properties of nested commutators}
Fix a flag simplical complex $\K$ on the vertex set $[m]=\{1,\dots,m\}$ and elements $q_1,\dots,q_m$ of a Lie algebra $L$.

\subsection{General properties}
The proof of the next lemma is similar to the proof of \cite[Lemma 4.7]{pv}.
\begin{lemma}
\label{l:ordering cwrs}
    Each nested commutator without repeats $a=[q_{i_1},[\dots[q_{i_k},q_j]\dots]]$ is equal to a Lie polynomial on nested commutators of the form $c_q(I,q_i)$, where $I\neq\varnothing$ and $\max(I)>i$.
\end{lemma}
\begin{proof}
    Induction on $k$. The base case: for $k=1$ we have $a=[q_{i_1},q_j]=c_q(\{i_1\},q_j)$ for $i_1>j$ and $a=\pm c_q(\{j\},q_{i_1})$ for $j>i_1$.

    Now let $k\geq 2$. For $t=1,\dots,k-2$ by the Jacobi identity we have
    \begin{align*}
    a=[q_{i_1},[\dots[q_{i_t},[q_{i_{t+1}},x]]\dots]]=&\pm[q_{i_1},[\dots[q_{i_{t+1}},[q_{i_t},x]]\dots]]\\
    &\pm[q_{i_1},[\dots[[q_{i_t},q_{i_{t+1}}],x]]\dots]].
    \end{align*}
    In the first summand the indices $i_t$ and $i_{t+1}$ are swapped. The second summand is of the form $c_{q,V}(I,[[q_{i_t},q_{i_{t+1}}],x])$ , so by Lemma \ref{l:c(I,[x,y])} it is a Lie polynomial on nested commutators without repeats which have smaller length (hence the inductive step applies). Thus we can freely permute $i_1,\dots,i_{k-1}$. Also, the identity $[q_{i_k},q_j]=\pm[q_j,q_{i_k}]$ allows to swap $i_k$ and $j$.

    After such a permutation we can assume that $i_1<\dots<i_{k-1}$ and $i_k>j$. If $i_{k-1}<i_k$, then $a$ is of the required form $c_q(\{i_1,\dots,i_k\},q_j)$. Otherwise $i_{k-1}>i_k$, so $i_{k-1}=\max\{i_1,\dots,i_k,j\}$. Use the identity
    $$[q_{i_{k-1}},[q_{i_k},q_j]]=\pm[q_{j},[q_{i_{k-1}},q_{i_k}]]\pm[q_{i_k},[q_{i_{k-1}},q_j]].$$
    We obtain a sum of two elements of the form $[q_{t_1},[\dots[q_{t_k},q_s]\dots],$
    where $t_k$ is the largest index. By permuting $t_1,\dots,t_{k-1}$ we can achieve that $t_1<\dots<t_k>s$. This expression has the required form $c_q(\{t_1,\dots,t_k\},q_{s})$.
\end{proof}
Recall that a Lie polynomial is a linear combination of iterated Lie brackets (of a given set of elements).
\begin{proposition}
\label{p:gptw generate cwr}
    Suppose that the elements $q_1,\dots,q_m\in L$ satisfy $[q_i,q_j]=0$ for any $\{i,j\}\in\K$. Then each nested commutator without repeats is a Lie polynomial on the GPTW generators.
\end{proposition}
\begin{proof}
By Lemma \ref{l:ordering cwrs}, it is sufficient to express the elements of the form $c_q(J\setminus j,q_j)$, $j\in J$, through the GPTW generators. In the case of Lie superalgebras such an expression exists by \cite[Theorem 4.3]{gptw} (see also a more explicit algorithm \cite[Algorithm 5.4]{vylegzhanin}). The Lie algebra case is similar up to a change of signs (instead of the identities $[x,y]=-(-1)^{\deg(x)\cdot\deg(y)} [y,x]$ we use $[x,y]=-[y,x]$, and similarly with the Jacobi identity.)
\end{proof}
\begin{remark}
    In the proofs of Lemma \ref{l:ordering cwrs} and Proposition \ref{p:gptw generate cwr} we used that the bracket is bilinear, commutative up to a sign and satisfies the Jacobi identity up to signs, without specifying the signs, and never used the identity $[x,x]=0$. So the same statements hold for Lie superalgebras and for elements of homotopy groups of topological spaces with respect to the Whitehead bracket \cite[\S X.7]{whitehead}.
\end{remark}

\subsection{The squaring operation}
\begin{lemma}
\label{lmm:surjective-squaring}
    Let $f:G_1\to G_2$ be a surjective map of groups, and suppose that the square of each element $x\in\gamma_k(G_1)$ is contained in $\gamma_{k+1}(G_1)$. Then the square of each element $y\in\gamma_k(G_2)$ is contained in $\gamma_{k+1}(G_2)$.
\end{lemma}
\begin{proof}    
    Since $\gamma_1(G)=G$ and $\gamma_{n+1}(G)=(G,\gamma_n(G))$, by induction we obtain that $\gamma_n(G_1)\to\gamma_n(G_2)$ is surjective for all $n$. Now let $y\in\gamma_k(G_2)$. Then $y=f(x)$ for some $x\in\gamma_k(G_1)$. By assumption, $x^2\in\gamma_{k+1}(G_1)$, so $y^2=f(x^2)\in\gamma_{k+1}(G_2)$.
\end{proof}

A version of the following operation is used implicitly in Prener's PhD thesis \cite{prener}. We discovered its properties inspired by an analogy with the composition product in homotopy groups (see Remark \ref{r:h_homotopy} below).
\begin{theorem}
    \label{t:h_properties}
    Let $G$ be a group generated by elements $g_1,\dots,g_m$ such that $g_i^2=1$ for all $i=1,\dots,m$.
    For $k\geq 2$ consider the map
    $$h:\gamma_k(G)/\gamma_{k+1}(G)\to\gamma_{k+1}(G)/\gamma_{k+2}(G),~a\gamma_{k+1}(G)\mapsto a^2\gamma_{k+2}(G).$$
    Then
    \begin{enumerate}
        \item $h:L_k(G)\to L_{k+1}(G)$ is a well defined linear map;
        \item $[x,h(y)]=h([x,y])$ for $x\in L_k(G)$, $y\in L_\ell(G)$, $\ell\geq 2$;
        \item $h([\og_i,y])=[\og_i,h(y)]$ for $i=1,\dots,m$, $y\in L_\ell(G)$, $\ell\geq 2$;
        \item $h([\og_i,\og_j])=[\og_i,[\og_i,\og_j]].$
    \end{enumerate}
\end{theorem}
\begin{proof}
    We use the shortened notation $\gamma_k:=\gamma_k(G)$. Recall that $\gamma_{k+1}\subset\gamma_k$, and for $a\in\gamma_k$, $b\in\gamma_\ell$ we have $(a,b)\in\gamma_{k+\ell}$. The Lie algebra operations in $L(G)=\bigoplus_k\gamma_k/\gamma_{k+1}$ are defined by the formulas     $\overline{a+b}=\overline{a}+\overline{b}$ and $[\overline{a},\overline{b}]=\overline{(a,b)}$, i.e. $a\gamma_{k+1}+b\gamma_{k+1}=ab\gamma_{k+1}$ and $[a\gamma_{k+1},b\gamma_{\ell+1}]=(a,b)\gamma_{k+\ell+1}.$

    By assumption, there is a surjective homomorphism $\RC_\K\to G$, where $\RC_\K=F(g_1,\dots,g_m)/(g_i^2=1,~i=1,\dots,m)$ is the right-angled Coxeter group corresponding to a simplicial complex with no edges. 
    By Proposition \ref{p:liz2} and the Lemma \ref{lmm:surjective-squaring}, the following property holds:
    \begin{equation}
    \label{eq:property}
    \text{If }c\in\gamma_k,\text{ then }c^2\in\gamma_{k+1}.
    \end{equation}
    
    1) We prove that $h$ is well defined. By \eqref{eq:property}, each $a\in\gamma_k$ determines an element $a^2\gamma_{k+2}$ of $\gamma_{k+1}/\gamma_{k+2}$. Now let $a,b\in\gamma_k$ and $a\gamma_{k+1}=b\gamma_{k+1}$, i.e. $b^{-1}a\in\gamma_{k+1}$; we need to check that $a^2\gamma_{k+2}=b^2\gamma_{k+2}$, i.e. that $b^{-2}a^2\in\gamma_{k+2}$. We have
    $$b^{-2}a^2=(b^{-1}a)^2a^{-1}ba^{-1}b^{-1}a^2=(b^{-1}a)^2 a^{-1}(b^{-1},a)a=(b^{-1}a)^2(a,(a,b^{-1}))(a,b^{-1}).
    $$
    Since $b^{-1}a\in\gamma_{k+1}$, we have $(b^{-1}a)^2\in\gamma_{k+2}$ by \eqref{eq:property}. Also $a,b^{-1}\in\gamma_k$ and $k\geq 2$, so $(a,(a,b^{-1}))\in\gamma_{3k}\subset\gamma_{k+2}$ and $(a,b^{-1})\in\gamma_{2k}\subset\gamma_{k+2}$. Hence $b^{-2}a^2\in\gamma_{k+2}$, so $h$ is well defined. Now we prove that $h$ is linear. Let $a,b\in \gamma_{k}$. We have
    \begin{gather*}h(\overline{a})+h(\overline{b})=a^2\gamma_{k+2}+ b^2\gamma_{k+2}=a^2b^2\gamma_{k+2}=\\aba(a,b)b\gamma_{k+2}=
    abab(a,b)((a,b),b)\gamma_{k+2}.
    \end{gather*}
    Since $k\geq 2$, we have $(a,b)\in\gamma_{2k}\subset\gamma_{k+2}$, $((a,b),b)\in\gamma_{3k}\subset\gamma_{k+2}$. Hence
    $$h(\overline{a})+h(\overline{b})=(ab)^2\gamma_{k+2}=h(\overline{ab})=h(\overline{a}+\overline{b}).
    $$
    2) By assumption, $x=\overline{a},$ $y=\overline{b}$ for some $a\in \gamma_k$, $b\in \gamma_\ell$, $\ell\geq 2$. We have
    $$[\overline{a},h(\overline{b})]=[a\gamma_{k+1}, b^2\gamma_{\ell+2}]=(a,b^2)\gamma_{k+\ell+2}=(a,b)(a,b)((a,b),b)\gamma_{k+\ell+2}.
    $$
     Since $\ell\geq 2$, we have $((a,b),b)\in\gamma_{k+2\ell}\subset\gamma_{k+\ell+2}$, so
    $$[\overline{a},h(\overline{b})]=(a,b)^2\gamma_{k+\ell+2}=h(\overline{(a,b)})=h([\overline{a},\overline{b}]).
    $$
    3) This is (2) for $x=\og_i$.\\
    4) We have $g_i^2=1$, so $(g_i,g_j)^2=(g_i,(g_i,g_j))$ in $G$. Hence
    \[h([\og_i,\og_j])=h((g_i,g_j)\gamma_3)=(g_i,g_j)^2\gamma_4=(g_i,(g_i,g_j))\gamma_4
    =[\og_i,[\og_i,\og_j]].\qedhere
    \]
\end{proof}
Recall that $L'(G)=\bigoplus_{k\geq 2}L_k(G)$.
\begin{definition}
Let $G$ be a group generated by elements $g_1,\dots,g_m$ such that $g_i^2=1$ for all $i=1,\dots,m$. We define the linear map
$$h:L'(G)\to L'(G),\quad L_k(G)\to L_{k+1}(G),\quad a\gamma_{k+1}(G)\mapsto a^2\gamma_{k+2}(G)$$
as a direct sum of linear maps constructed in Theorem \ref{t:h_properties}. By part (2) of the theorem, it  satisfies the identity $h([x,y])=[h(x),y]=[x,h(y)]$ for any $x,y\in L'(G)$.
\end{definition}
Note that the definition above applies to all right-angled Coxeter groups $\RC_\K$. In this paper, we will consider only the case $G=\RC_\K$.

\begin{remark}
\label{r:h_homotopy}
    Let $X$ be a topological space and $f:S^{k+1}\to X$ be a continuous map for some $k\geq 2$. Consider the composite map $h(f):=f\circ E^{k-1}\eta:S^{k+2}\to X$, where $E^{k-1}\eta:S^{k+2}\to S^{k+1}$ is the $(k-1)$-fold suspension over the Hopf fibration $\eta:S^3\to S^2$. We obtain an operation $h:\pi_{k+1}(X)\to\pi_{k+2}(X)$ on homotopy groups, which satisfies the analogues of properties (1)-(3) with respect to the Whitehead bracket $\pi_{k+1}(X)\times\pi_{\ell+1}(X)\to\pi_{k+\ell+1}(X)$, see \cite[Ch. X, (8.2),~(8.18),~(8.9)]{whitehead}. The property (4) is analogous to the identity $[a_1,a_2]\circ E\eta=[a_1,[a_1,a_2]]$ which holds for the standard generators $a_1,a_2\in\pi_2(\CC P^\infty\vee\CC P^\infty)\simeq\ZZ^2$ by \cite[Lemma 8.7]{kallel}.
\end{remark}
\subsection{Computations in the associated Lie algebra}
\begin{lemma}
\label{l:h_commutator_outer_letter}
    Let $x=[\og_{i_1},[\dots[\og_{i_{k}},\og_{j}]\dots]]\in L(\RC_\K)$ be a nested commutator, $k\geq 1$. Then $[\og_{i_1},x]=h(x).$
\end{lemma}
\begin{proof}
    Induction on $k.$ The base case $[\og_i,[\og_i,\og_j]]=h([\og_i,\og_j])$ is the part (4) of Theorem \ref{t:h_properties}. The inductive step: without loss of generality, $i_1=1$ and $i_2=2$. Then $x=[\og_1,[\og_2,y]]$, and $[\og_1,[\og_1,y]]=h([\og_1,y])$ by the inductive assumption. By the Jacobi identity and Theorem \ref{t:h_properties}, we have
    \begin{multline*}
    [\og_1,x]=[\og_1,[\og_1,[\og_2,y]]]=[\og_1,[[\og_1,\og_2],y]]+[\og_1,[\og_2,[\og_1,y]]]\\
    =[[\og_1,[\og_1,\og_2]],y]+\underbrace{[[\og_1,\og_2],[\og_1,y]]+[[\og_1,\og_2],[\og_1,y]]}_{=0}+[\og_2,[\og_1,[\og_1,y]]]\\
    =[h([\og_1,\og_2]),y]+[\og_2,h([\og_1,y])]\\
    =h([[\og_1,\og_2],y]+[\og_2,[\og_1,y]])=h([\og_1,[\og_2,y]])=h(x).\qedhere
    \end{multline*}
\end{proof}

\begin{corollary}
\label{c:h_commutator_inner_letter}
    Let $y=c_{\og,V}(I,[\og_i,x])$, where $x$ is a nested commutator or $x=\og_j$. Then
    $$[\og_i,y]=h(y)+\sum_{\begin{smallmatrix}
        I=A\sqcup B,\\
        A\neq\varnothing
    \end{smallmatrix}}
    \Big[c_{\og,V}(A,\og_i),c_{\og,V}(B,[\og_i,x])\Big].
    $$
\end{corollary}
\begin{proof}
    By Lemma \ref{l:h_commutator_outer_letter}, we have $h([\og_i,x])=[\og_i,[\og_i,x]]$. Hence
$$h(y)=h(c_{\og,V}(I,[\og_i,x]))=c_{\og,V}(I,[\og_i,[\og_i,x]])=\sum_{I=A\sqcup B}\Big[c_{\og,V}(A,\og_i),c_{\og,V}(B,[\og_i,x])\Big]$$
by Lemma \ref{l:c(I,[x,y])} and Theorem \ref{t:h_properties}. The case $A=\varnothing$ correspond to the summand $[\og_i,c_{\og,V}(I,[\og_i,x])]=[\og_i,y]$.
\end{proof}
\begin{proposition}
\label{p:removing repeats}
    Let $a\in L'(\RC_\K)$ be a nested commutator of $\og_1,\dots,\og_m$. Then it can be written as a sum $a=\sum_s h^{n_s}(a_s')$, where $n_s\geq 0$, and each element $a_s'$ is a Lie polynomial on nested commutators without repeats.
\end{proposition}
\begin{proof}
Induction on length of $a$. The base case is $a=[\og_i,\og_j]$, which is zero if $i=j$ and has no repeats if $i\neq j$.

The inductive step: if $a$ has no repeats, then $a=h^0(a)$ and we are done. Otherwise $a$ can be written as
$$a=c_{\og,W}(J,[\og_i,c_{\og,V}(I,[\og_i,x])]),$$
where $x$ is a nested commutator or $x=\og_j$, while $W$ and $V$ are some sequences of elements from $[m]$.
By Corollary \ref{c:h_commutator_inner_letter}, Lemma \ref{l:c(I,[x,y])} and properties of $h$, we obtain:
\begin{multline*}
a=h(c_{\og,W}(J,c_{\og,V}(I,[\og_i,x])))
\\
+\sum_{J=C\sqcup D}
\sum_{\begin{smallmatrix}
I=A\sqcup B,\\
A\neq\varnothing
\end{smallmatrix}}
\Big[
c_{\og,W}(C,c_{\og,V}(A,\og_i)),
c_{\og,W}(D,c_{\og,V}(B,[\og_i,x]))
\Big].
\end{multline*}
The right hand side is expressed in terms of nested commutators
$$c_{\og,W}(J,c_{\og,V}(I,[\og_i,x])),\quad c_{\og,W}(C,c_{\og,V}(A,\og_i)),\quad c_{\og,W}(D,c_{\og,V}(B,[\og_i,x]))$$
of smaller length. Applying the inductive assumption to them and using the properties of $h$, we obtain the required expression for $a$.
\end{proof}
\section{The surjective homomorphism onto the commutator subalgebra}
For a flag simplicial complex $\K$ on the vertex set $[m]$, consider the GPTW elements as a subset of $L_\K$,
$$\GPTW:=\{c_\mu(J\setminus j,\mu_j):j\in \Theta_\K(J),J\subset[m]\}\subset L_\K.$$

\begin{definition}
    Denote by $N_\K$ the Lie subalgebra of $L_\K$ generated by the GPTW generators,
    $$N_\K:=\langle\GPTW\rangle_{Lie}\subset L_\K.$$
\end{definition}
By definition, $N_\K\subset L'_\K.$ This inclusion is strict unless $\K$ is a simplex, since $[\mu_i,[\mu_i,\mu_j]]\in L'_\K$ does not belong to $N_\K$ if $\{i,j\}\notin\K.$

\begin{definition}
    Let $\mathfrak{g}=\bigoplus_{n\in\ZZ}\mathfrak{g}_n$ be a graded Lie algebra. We introduce a graded Lie algebra structure on the polynomial ring $\mathfrak{g}[t]=\bigoplus_{k\geq 0}\mathfrak{g}t^k$ as follows:
    \begin{itemize}
        \item $\deg(xt^k):=n+k$ for $x\in\mathfrak{g}_n$;
        \item $[xt^k,yt^\ell]:=[x,y]t^{k+\ell}$ for $x,y\in\mathfrak{g}$.
    \end{itemize}
\end{definition}

Since $N_\K\subset L'_\K$, the epimorphism $\varphi:L_\K\to L(\RC_\K)$, $\mu_i\mapsto\og_i$ from Proposition \ref{p:epinmono} can be restricted to a map of Lie algebras
$$\varphi':N_\K\to L'(\RC_\K).$$ 
\begin{theorem}
\label{t:main_surjective_homomorphism}
    We have a surjective homomorphism of graded Lie algebras
    $$\psi:N_\K[t]\twoheadrightarrow L'(\RC_\K),\quad \psi(xt^k):=h^k(\varphi'(x)),~x\in N_\K.$$
\end{theorem}
\begin{proof}
The map $\psi$ is linear since $\varphi'$ is a map of Lie algebras, while $h:L'(\RC_\K)\to L'(\RC_\K)$ is linear by Theorem \ref{t:h_properties}. Also, $\psi$ is a map of Lie algebras: it maps a bracket $[xt^k,yt^\ell]=[x,y]t^{k+\ell}$ to the element
$$h^{k+\ell}(\varphi'([x,y]))=\Big[h^k(\varphi'(x)),h^\ell(\varphi'(y))\Big]=[\psi(xt^k),\psi(yt^\ell)],$$
since $h(-)$ satisfies the identity $h([x,y])=[h(x),y)]=[x,h(y)].$

It remains to check that $\psi$ is onto. Let $n\geq 2$. By \cite[Proposition 3.1, Corollary 3.2]{veryovkin}, the abelian group $L_n(\RC_\K)$ is generated by nested commutators of length $n$ on the elements $\og_1,\dots,\og_m$. By Proposition \ref{p:removing repeats}, such a commutator $a$ is a linear combination of elements of the form $h^{n_s}(a'_s)$, where $a'_s$ is a Lie polynomial on nested commutators without repeats. Such a nested commutator is a Lie polynomial on GPTW generators by Proposition \ref{p:gptw generate cwr}, so $a'_s$ is a Lie polynomial on $\GPTW$. In other words, $a'_s=\varphi(x_s)$ for some $x_s\in N_\K$, so $h^{n_s}(a'_s)\in\Img\psi$ and $a\in\Img\psi$.
\end{proof}
\begin{conjecture}
\label{cnj:main_conjecture}
    The surjective map of Lie algebras $\psi:N_\K[t]\twoheadrightarrow L'(\RC_\K)$ is an isomorphism for any flag complex $\K$.
\end{conjecture}

The additive and multiplicative structure of the Lie algebra $N_\K$ over $\ZZ_2$ will be described in Section \ref{sec:N_K} using the connection with loop homology of polyhedral products, see Proposition \ref{p:N_K_poincare_series} and Theorem \ref{t:N_K_presentation}. 
This allows to give a restatement of Conjecture \ref{cnj:main_conjecture} in terms of the numbers $\dim_{\ZZ_2} L_k(\RC_\K)$, see Proposition \ref{p:conjectural_poincare_series}.

\section{Polyhedral products and loop homology}
\label{section:polyhedral-products}
In this sections we prove auxiliary facts about loop homology of some topological spaces, which are used in Section \ref{sec:N_K} to study the Lie algebra $N_\K$ over $\ZZ_2$.

\begin{definition}
Let $\K$ be a simplicial complex on the vertex set $[m]$ and $(X,A)$ be a pair of topological spaces. The corresponding \emph{polyhedral product} is the space
$$(X,A)^\K:=\bigcup_{I\in\K}Y_1\times\dots\times Y_m\subset X^m,\quad Y_i=\begin{cases}X,&i\in I;\\ A,&i\notin I.\end{cases}$$
In toric topology, two special cases of this construction are important: the \emph{moment-angle comples} $\ZK:=(D^2,S^1)^\K$ and the \emph{Davis--Januszkiewicz space} $\DJ(\K):=(\CC P^\infty,\ast)^\K$ (here $\ast\subset \CC P^\infty$ is the basepoint).
\end{definition}
The polyhedral product construction is functorial with respect to maps of pairs of topological spaces and with respect to inclusions of simplicial complexes \cite[Proposition 4.2.3]{bu-pa15}. For more details and examples, see~\cite[\S4.3]{bu-pa15},~\cite{bbc}.

\subsection{Whitehead brackets in polyhedral products}
\begin{proposition}[{\cite[Corollary 3.14, Lemma 3.16]{toric_homotopy}}]
\label{p:fibration}
Let $\K$ be a simplicial complex on $m$ vertices and $X$ be a space with the basepoint $\ast$. Then we have a natural homotopy fibration $$(C\Omega X,\Omega X)^\K\overset\iota\longrightarrow (X,\ast)^\K\overset s\longrightarrow X^m,$$ where $s$ is the standard inclusion (induced by the inclusion $\K\hookrightarrow \Delta^{m-1}$ of simplicial complexes), $C$ is the cone and $\Omega X$ is the based loop space of $X$. Moreover, the map $\Omega s$ has a homotopy section.\qed
\end{proposition}

Since the suspension $\Sigma$ and the loop space $\Omega$ are adjoint functors, there are natural maps $E_X:X\to\Omega\Sigma X$ and $ev_X:\Sigma\Omega X\to X$ such that the diagrams below commute up to a homotopy:
\begin{equation}
\label{eq:adjoint}
\xymatrix{
\Sigma X
\ar[rr]^\id
\ar[dr]_-{\Sigma E_X}
&&
\Sigma X,
\\
&\Sigma\Omega\Sigma X
\ar[ur]_-{ev_{\Sigma X}}
}
~
\xymatrix{
\Omega X
\ar[rr]^\id
\ar[dr]_-{E_{\Omega X}}
&&
\Omega X.
\\
&\Omega\Sigma\Omega X
\ar[ur]_-{\Omega ev_{X}}
}
\end{equation}

From naturality of the polyhedral product construction we obtain the following diagram which commutes up to a homotopy (see \cite[(20)]{panov_theriault}):
\begin{equation}
\label{eq:diagram of spaces}
\xymatrix{
\ZK
\ar[d]^-{g}
\ar@/_4pc/[dd]_-{\id}
\\
(C\Omega S^2,\Omega S^2)^\K
\ar[d]^-{G}
\ar[r]^-{\iota_1}
&
(S^2,\ast)^\K
\ar[d]^-{F}\\
\ZK
\ar[r]^-{\iota_2}
&
\DJ(\K)
}
\end{equation}
Here the horizontal maps are from Proposition \ref{p:fibration}, and the maps
$$
\ZK=(C S^1,S^1)^\K\overset{g}\longrightarrow (C\Omega S^2,\Omega S^2)^\K,~
(C\Omega S^2,\Omega S^2)^\K\overset{G}\longrightarrow(CS^1,S^1)^\K
$$
are induced by the maps
$$E_{S^1}:S^1\to \Omega S^2,\quad \Omega ev_{\CC P^\infty}:\Omega S^2\to\Omega\CC P^\infty\simeq S^1.$$
We have $\Omega ev_{\CC P^\infty}\circ E_{S^1}=\id:S^1\to S^1$ by \eqref{eq:adjoint}, hence $G\circ g=\id:\ZK\to\ZK$.

For $j=1,\dots,m$ consider the maps
$$incl_{X,j}:X\hookrightarrow (X,\ast)^\K,\quad
ev_{X,j}:\Sigma\Omega X\overset{ev_X}\longrightarrow X\overset{incl_{X,j}}\longrightarrow (X,\ast)^\K$$
and their special cases
\begin{gather*}
incl_j=incl_{S^2,j}:\Sigma S^1=S^2\hookrightarrow (S^2,\ast)^\K,\quad
ev_j=ev_{S^2,j}:\Sigma\Omega S^2\to (S^2,\ast)^\K,\\
a_j=ev_{\CC P^\infty,j}:\Sigma S^1\cong \Sigma\Omega\CC P^\infty\to (\CC P^\infty,\ast)^\K=\DJ(\K).
\end{gather*}
It is well known that the map $ev_{\CC P^\infty}:S^2\simeq\Sigma\Omega \CC P^\infty\to\CC P^\infty$ is homotopic to the natural inclusion $S^2\hookrightarrow\CC P^\infty$, hence $a_j$ can be defined as the composition
$$a_j:\Sigma S^1=S^2\hookrightarrow\CC P^\infty\hookrightarrow (\CC P^\infty,\ast)^\K.$$

For $j\in J\subset[m]$, $J\setminus j=\{i_1<\dots<i_k\}$, $k\geq 1$,  consider the iterated generalised Whitehead products
\begin{gather*}
c_{ev_X}(J\setminus j,ev_{X,j}):=[ev_{X,i_1},[\dots[ev_{X,i_k},ev_{X,j}]\dots]]:\Sigma(\Omega X)^{\wedge k+1}\to (X,\ast)^\K,\\
c_{incl}(J\setminus j,incl_j):=[incl_{i_1},[\dots[incl_{i_k},incl_j]\dots]]:\Sigma(S^1)^{\wedge k+1}\to (S^2,\ast)^\K,\\
c_{ev}(J\setminus j,ev_j):=[ev_{i_1},[\dots[ev_{i_k},ev_j]\dots]]:\Sigma (\Omega S^2)^{\wedge k+1}\to (S^2,\ast)^\K,\\
c_a(J\setminus j,a_j):=[a_{i_1},[\dots[a_{i_k},a_j]\dots]]:\Sigma(S^1)^{\wedge k+1}\to \DJ(\K).
\end{gather*}
When $j$ and $J$ are clear from the context, we just write $c_{ev_X}$, $c_{incl}$, $c_{ev}$, $c_a$. Here $X^{\wedge n}$ is the smash product of $n$ copies of $X$, e.g. $(S^1)^{\wedge n}\cong S^n$. 

The following result refines the classical theorem of Porter \cite{porter}.
\begin{theorem}[{\cite[Theorem 7.2]{theriault_cocat}, see also \cite[Corollary 5.9]{iriye_kishimoto}}]
\label{t:porter}
Let $\L$ be the disjoint union of $m$ points and $X$ be a simply connected space.
Then there is a natural homotopy equivalence
$$\gamma_\L:\bigvee_{J\subset[m],J\neq\varnothing}\bigvee_{j\in J\setminus\{\max(J)\}}\Sigma(\Omega X)^{\wedge |J|}\to(C\Omega X,\Omega X)^\L$$
such that the following diagram is homotopy commutative:
$$\xymatrix{
\bigvee_{J\subset[m],J\neq\varnothing}\bigvee_{j\in J\setminus\{\max(J)\}}\Sigma(\Omega X)^{\wedge |J|}
\ar@/^/[drr]^-{\quad\vee_{J,j}c_{ev_X}(J\setminus j,ev_{X,j})}
\ar[d]^-{\gamma_\L}
&&
\\
(C\Omega X,\Omega X)^\L
\ar[rr]^-\iota
&&
(X,\ast)^\L.\qed
}$$
\end{theorem}
Let $j\in J\subset [m]$, $J\setminus j=\{i_1<\dots<i_k\},$ $k\geq 1$. Let $\K$ be a simplicial complex on $[m]$. Consider the composite maps
\begin{gather*}
\gamma_1:\Sigma(\Omega S^2)^{\wedge k+1}\to (C\Omega S^2,\Omega S^2)^\L\to (C\Omega S^2,\Omega S^2)^\K,\\
\gamma_2:\Sigma(S^1)^{\wedge k+1}\to (CS^1,S^1)^\L\to(CS^1,S^1)^\K,
\end{gather*}
where the first arrows are the restrictions of $\gamma_\L$ corresponding to $c_{ev_X}(J\setminus j,ev_{X,j})$ for $X=S^2$ and $X=\CC P^\infty$, respectively, and the second arrows are induced by the inclusion $\L\hookrightarrow\K$.
\begin{lemma}
\label{l:some strange compositions}
The map $\gamma_2:\Sigma (S^1)^{\wedge k+1}\to\ZK$ satisfies
\begin{gather*}
\iota_2\circ\gamma_2=c_a(J\setminus j,a_j):\Sigma(S^1)^{\wedge k+1}\to\DJ(\K),\\
\iota_1\circ g\circ\gamma_2=c_{incl}(J\setminus j,incl_j):\Sigma(S^1)^{\wedge k+1}\to (S^2,\ast)^\K.
\end{gather*}
\end{lemma}
\begin{proof}
By naturality, it is sufficient to consider the case $\K=\L$.
Since $a_j=ev_{\CC P^\infty,j}$, the identity $\iota_2\circ\gamma_2=c_a$ follows from Theorem \ref{t:porter}. To prove the second identity, we will show that the diagram below is homotopy commutative.
$$
\xymatrix{
\Sigma(S^1)^{\wedge k+1}
\ar[d]^-{\Sigma E_{S^1}^{\wedge k+1}}
\ar[r]^{\gamma_2}
\ar@/^1.5pc/[rr]^-{c_{incl}}
&
(CS^1,S^1)^\L
\ar[d]^-g
&
(S^2,\ast)^\L
\ar@{=}[d]
\\
\Sigma(\Omega S^2)^{\wedge k+1}
\ar[r]^-{\gamma_1}
\ar@/_1.5pc/[rr]_-{c_{ev}}
&
(C\Omega S^2,\Omega S^2)^\L
\ar[r]^-{\iota_1}&
(S^2,\ast)^\L\\
}
$$
Consider the map $ev_j\circ \Sigma E_{S^1}:\Sigma S^1\to\Sigma\Omega S^2\to (S^2,\ast)^\L$ for some $j\in[m]$.
It is a composition
$$\xymatrix{
\Sigma S^1
\ar[r]^-{\Sigma E_{S^1}}
&
\Sigma\Omega S^2
\ar[r]^-{ev_{S^2}}
&
S^2
\ar[r]^-{incl_j}
&
(S^2,\ast)^\L,
}$$
but $ev_{S^2}\circ\Sigma E_{S^1}$ is the identity map by the left diagram in \eqref{eq:adjoint}. Hence $ev_j\circ \Sigma E_{S^1}=incl_j$. Taking Whitehead products, we obtain $c_{ev}\circ \Sigma E_{S^1}^{\wedge k+1}=c_{incl}$. We also have $\iota_1\circ\gamma_1=c_{ev}$ by Theorem \ref{t:porter} and $g\circ\gamma_2=\gamma_1\circ \Sigma E_{S^1}^{\wedge k+1}$ by naturality of $\gamma_\L$. So
\[\iota_1\circ g\circ\gamma_2=\iota_1\circ\gamma_1\circ \Sigma E_{S^1}^{\wedge k+1}=c_{ev}\circ \Sigma E_{S^1}^{\wedge k+1}=c_{incl}.\qedhere\]
\end{proof}

\subsection{Homomorphisms of loop homology algebras}
Until the end of this section, fix a field $\k$.
\begin{proposition}
\label{p:hopf_extension}
Let $\K$ be a simplicial complex on $m$ vertices and $X$ be a space with the basepoint $\ast$. Then there is a natural homotopy equivalence $\Omega(X,\ast)^\K\simeq \Omega X^m\times \Omega(C\Omega X,\Omega X)^\K$. For any field $\k$, the induced maps in homology
$$H_*(\Omega(C\Omega X,\Omega X)^\K;\k)\overset{\iota_*}\longrightarrow H_*(\Omega(X,\ast)^\K;\k)\to H_*(\Omega X^m;\k)$$
 form an extension of Hopf algebras. In particular, there is a natural isomorphism of $\k$-modules $$H_*(\Omega(X,\ast)^\K;\k)\simeq H_*(\Omega(C\Omega X,\Omega X)^\K;\k)\otimes_\k H_*(\Omega X^m;\k),$$ and the homomorphism $\iota_*:H_*(\Omega(C\Omega X,\Omega X)^\K;\k)\to H_*(\Omega(X,\ast)^\K;\k)$ is injective.
\end{proposition}
\begin{proof}
Apply
\cite[Lemma B.2, Theorem B.3]{vylegzhanin} to the fibration from Proposition \ref{p:fibration}.
\end{proof}

Applying the loop homology functor to \eqref{eq:diagram of spaces}, we obtain the following diagram of Hopf algebras and their homomorphisms (where $\hookrightarrow$ denote injective maps):
\begin{equation}
\label{eq:diagram_of_algebras}
\xymatrix{
H_*(\Omega\ZK;\k)
\ar[d]^-{g_*}
\ar@/_6pc/[dd]_-{\id}
\\
H_*(\Omega(C\Omega S^2,\Omega S^2)^\K;\k)
\ar[d]^-{G_*}
\ar@{^(->}[r]^-{\iota_{1*}}
&
H_*(\Omega(S^2,\ast)^\K;\k)
\ar[d]^-{F_*}\\
H_*(\Omega\ZK;\k)\ar@{^(->}[r]^-{\iota_{2*}}
&
H_*(\Omega\DJ(\K);\k).
}
\end{equation}
Thus $g_*$ is injective. Also, the maps $\iota_{1*},$ $\iota_{2*}$ are injective by Proposition \ref{p:hopf_extension}.

For homogeneous elements $x,y$ of a graded associative algebra $C$, their graded commutator is defined as $[x,y]:=xy-(-1)^{\deg(x)\cdot\deg(y)}yx\in C$.  We consider the Hurewicz map
\begin{equation}
\label{eq:hurewicz}
\eta:\pi_{*+1}(X)\to H_*(\Omega X;\k),
\end{equation}
choosing the sign in such a way that $\eta$ takes Whitehead products to graded commutators \cite{samelson}.

From now on, $\K$ is a flag complex on $[m]$. Consider the graded Hopf algebras
\begin{align*}
A&:=T(x_1,\dots,x_m)/(x_ix_j+x_jx_i=0,\{i,j\}\in\K),\\
B&:= T(u_1,\dots,u_m)/(u_iu_j+u_ju_i=0,\{i,j\}\in\K;~u_i^2=0,~i=1,\dots,m).
\end{align*}
\begin{proposition}
\label{p:loop_homology}
Let $\K$ be a flag complex on $[m]$ and $\k$ be a field. Then we have natural isomorphisms of graded Hopf algebras
$$A\cong H_*(\Omega(S^2,\ast)^\K;\k),\quad B\cong H_*(\Omega\DJ(\K);\k).$$
Under this isomorphisms, the maps from the diagram \eqref{eq:diagram_of_algebras} satisfy $F_*(x_i)=u_i$ and $$x_i=\eta(incl_i)\in H_1(\Omega(S^2,\ast)^\K;\k),\quad u_i=\eta(a_i)\in H_1(\Omega\DJ(\K);\k).$$ These algebras admit a $\Zm$-grading $\deg u_i=\deg x_i=e_i\in\Zm$.
\end{proposition}
\begin{proof}
See \cite[Theorem 1.3]{bubenik_gold} and \cite[Theorem 9.3]{panov_ray}.
\end{proof}
For an extension of Hopf algebras $R'\to R\overset{\pi}\longrightarrow R''$, $R'$ is determined by $\pi$ as $R'=\{x\in R:~((\id_{R''}\otimes\Delta)(\pi(x))=x\otimes 1\}$ (see e.g. \cite[Definition 2.3]{vylegzhanin22}). It follows that the Hopf subalgebras $$H_*(\Omega(C\Omega S^2,\Omega S^2)^\K;\k)\hookrightarrow H_*(\Omega(S^2,\ast)^\K;\k)\cong A,\quad H_*(\Omega\ZK;\k)\hookrightarrow H_*(\Omega\DJ(\K);\k)\cong B$$ and the morphism $G_*:H_*(\Omega(C\Omega S^2,\Omega S^2)^\K;\k)\to H_*(\Omega\ZK;\k)$ admit an algebraic description:
\begin{corollary}
\label{c:map_of_hopf_extensions}
There is a map of Hopf algebra extensions
$$\xymatrix{
H_*(\Omega(C\Omega S^2,\Omega S^2)^\K)
\ar[r]^-{\iota_{1*}}
\ar[d]^-{G_*}
&
A
\ar[r]^-{\pi_1}
\ar[d]^-{F_*}
&
T(x_1,\dots,x_m)/(x_ix_j+x_jx_i=0,~i\neq j)
\ar[d]^-{x_i\mapsto u_i}
\\
H_*(\Omega\ZK)
\ar[r]^-{\iota_{2*}}
&
B
\ar[r]^-{\pi_2}
&
T(u_1,\dots,u_m)/(u_iu_j+u_ju_i=0;~u_i^2=0).\qed
}$$
\end{corollary}

For a subset $J\subset[m]$ and an element $j\in J$, $J\setminus j=\{i_1<\dots<i_k\}$, $k\geq 1$, consider the iterated graded commutators
\begin{gather*}
c'_x=c_x'(J\setminus j,x_j):=[x_{i_1},[\dots[x_{i_k},x_j]\dots]]\in A,\\
c'_u=c_u'(J\setminus j,u_j):=[u_{i_1},[\dots[u_{i_k},u_j]\dots]]\in B.
\end{gather*}

\begin{lemma}
\label{l:unique c_x and c_u}
In the notation above,
\begin{enumerate}
\item There is unique element $\widehat{c}_x(J\setminus j,x_j)\in H_*(\Omega(C\Omega S^2,\Omega S^2)^\K;\k)$ such that\\$\iota_{1*}(\widehat{c}_x)=c'_x\in A$;
\item There is unique element $\widehat{c}_u(J\setminus j,u_j)\in H_*(\Omega\ZK;\k)$ such that\\$\iota_{2*}(\widehat{c}_u)=c'_u\in B$.
\end{enumerate}
\end{lemma}
\begin{proof}
Consider the diagram from Corollary \ref{c:map_of_hopf_extensions}.
Since $\iota_{1*}$ and $\iota_{2*}$ are injective by Proposition \ref{p:hopf_extension}, it is sufficient to prove the existence. We follow the argument from \cite[Corollary 3.10]{vylegzhanin}:
the elements $x_i\in A$ and $u_i\in B$ are primitive by degree reasons, so $c'_x\in A$ and $c'_u\in B$ since the primitive elements of a Hopf algebra form a Lie algebra. Also, $c'_x\in\Ker \pi_1$ and $c'_u\in\Ker\pi_2$. The primitive elements of Hopf algebra extensions form an exact sequence of Lie algebras (see \cite[Corollary B.4]{vylegzhanin}), so $c'_x\in\Img\iota_{1*}$ and $c'_u\in\Img\iota_{2*}$.
\end{proof}

The next theorem is a purely algebraic result on homomorphisms of partially commutative Hopf algebras. Remarkably, it is obtained by topological arguments.
\begin{theorem}
\label{t:section_description}
Let $\K$ be a simplicial complex on the vertex set $[m]$ and $\k$ be a field.
Then the restriction $G_*:H_*(\Omega\ZK;\k)\to H_*(\Omega(C\Omega S^2,\Omega S^2)^\K;\k)$ of the homomorphism $F_*:A\to B,$ $x_i\mapsto u_i$, admits a $\Zm$-graded multiplicative section $$g_*:H_*(\Omega(C\Omega S^2,\Omega S^2)^\K;\k)\to H_*(\Omega\ZK;\k)$$ such that
$$g_*(\widehat{c}_u(J\setminus j,u_j))=\widehat{c}_x(J\setminus j,x_j)$$ for any $J\subset[m]$, $|J|\geq 2$, $j\in J$. In particular, $\iota_{1*}(g_*(\widehat{c}_u(J\setminus j,u_j)))=c_x'(J\setminus j,x_j)$.
\end{theorem}
\begin{proof}
The homomorphism $g_*$ of Hopf algebras is induced by the map $g:\ZK\to (C\Omega S^2,\Omega S^2)^\K$ from the diagram \eqref{eq:diagram of spaces}. It is a section of $G_*$ by the commutativity of the induced diagram \eqref{eq:diagram_of_algebras}. The fact that $g_*$ is $\Zm$-graded will be proved later.

We now prove that $g_*(\widehat{c}_u)=\widehat{c}_x.$ Recall that $\iota_2\circ\gamma_2=c_a$ and $\iota_1\circ g\circ\gamma_2=c_{incl}$ by Lemma \ref{l:some strange compositions}.
 Applying the Hurewicz map \eqref{eq:hurewicz}, we obtain $\iota_{2*}(\eta(\gamma_2))=c'_u$ and $\iota_{1*}(g_*(\eta(\gamma_2))=c'_x$ by Proposition \ref{p:loop_homology}. On the other hand, $c'_u=\iota_{2*}(\widehat{c}_u)$ and $c'_x=\iota_{1*}(\widehat{c}_x)$ by Lemma \ref{l:unique c_x and c_u}. Since $\iota_{1*},$ $\iota_{2*}$ are injective, we obtain
$$\eta(\gamma_2)=\widehat{c}_u,\quad g_*(\eta(\gamma_2))=\widehat{c}_x.$$
It follows that $g_*(\widehat{c}_u)=\widehat{c}_x$ and $\iota_{1*}(g_*(\widehat{c}_u))=\iota_{1*}(\widehat{c}_x)=c'_x$.

Finally, $g_*$ maps $\Zm$-homogeneous elements $\widehat{c}_u\in H_*(\Omega\ZK;\k)$ to $\Zm$-homogeneous elements $\widehat{c}_x$ of the same degree. The algebra $H_*(\Omega\ZK;\k)$ is generated by the elements $\widehat{c}_u$ by \cite[Theorem 4.3]{gptw}, so $g_*$ preserves the $\Zm$-grading.
\end{proof}
\begin{remark}
The elements $\widehat{c}_u$, $\widehat{c}_x$ have the same multidegree $\sum_{j\in J}e_j\in\Zm$, and $G_*(\widehat{c}_x)=\widehat{c}_u$. It can be shown that the algebras $H_*(\Omega(C\Omega S^2,\Omega S^2)^\K;\k)$ and $H_*(\Omega\ZK;\k)$ are of same dimension in this multidegree, so $G_*\circ g_*=\id$ implies $g_*(\widehat{c}_u)=\widehat{c}_x$. However, we do not know in advance that $g_*$ preserves the $\Zm$-grading, so this argument does not prove Theorem \ref{t:section_description}.
\end{remark}
\section{Structure of the Lie algebra \texorpdfstring{$N_\K$}{NK}}
\label{sec:N_K}

\subsection{Universal enveloping algebra and loop homology of moment-angle complexes}
\begin{lemma}
\label{l:universal_enveloping_generated}
    Let $L$ be a Lie algebra over a field $\k$ and $UL$ be its universal enveloping algebra. For a subset $X\subset L$, consider the Lie subalgebra $N:=\langle X\rangle_{Lie}\subset L$ and the associative subalgebra $A:=\langle X\rangle_{Ass}\subset UL$ generated by $X$. Then the natural map $UN\to A$ is an isomorphism.
\end{lemma}
\begin{proof}
    The inclusion of the Lie subalgebra $N\subset L$ induces an injective map of universal enveloping algebras (e.g. by the Poincar\'e--Birkhoff--Witt theorem); thus $UN\subset UL$ is a subalgebra. Since $UN=\langle N\rangle_{Ass}$ and $N=\langle X\rangle_{Lie}$, the associative algebra $UN$ is generated by $X$. Hence $UN=\langle X\rangle_{Ass}=A$.
\end{proof}

Recall that the subset $\GPTW=\{c_\mu(J\setminus j,\mu_j):J\subset[m],j\in\Theta_\K(J)\}$ of
$$L_\K=FL_{\ZZ_2}(\mu_1,\dots,\mu_m)/([\mu_i,\mu_j]=0,~\{i,j\}\in\K)$$
generates the Lie subalgebra $N_\K\subset L_\K$. 
\begin{theorem}
    \label{t:U(N_K)}
    For any flag complex $\K$ we have an isomorphism $$U(N_\K)\cong H_*(\Omega\ZK;\ZZ_2)$$
     of $\Zm$-graded associative $\ZZ_2$-algebras.
\end{theorem}
\begin{proof}
    For any field $\k$, we have an injective map of algebras
    $$\iota_{1*}\circ g_*:H_*(\Omega\ZK;\k)\to H_*(\Omega(S^2,\ast)^\K;\k)\cong A$$
    from the diagram \eqref{eq:diagram_of_algebras}.
    For $\k=\ZZ_2$ we also have an isomorphism
    \begin{align*}
    A=~&T_{\ZZ_2}(x_1,\dots,x_m)/(x_ix_j+x_jx_i=0,~\{i,j\}\in\K)\\
    \cong~&T_{\ZZ_2}(\mu_1,\dots,\mu_m)/(\mu_i\mu_j-\mu_j\mu_i=0,~\{i,j\}\in\K)\cong U(L_\K),
    \end{align*}
    which identifies the nested graded commutators $c_x'(I,x_j)\in A$ with the nested commutators $c_\mu(I,\mu_j)\in L_\K$. (Here we again use that $\k=\ZZ_2$.)

On the other hand, the algebra $H_*(\Omega\ZK;\k)$ is multiplicatively generated by the set $\{\widehat{c}_u(J\setminus j,u_j):J\subset[m],~j\in\Theta_\K(J)\},$ see \cite[Theorem 4.3]{gptw}. By Theorem \ref{t:section_description}, the map $\iota_{1*}\circ g_*$ takes this set to the set $\GPTW\subset A$. Hence $\Img(\iota_{1*}\circ g_*)=$ $\langle\GPTW\rangle_{Ass}\subset U(L_\K)$. But $N_\K:=\langle\GPTW\rangle_{Lie}\subset L_\K$, so we obtain the required isomorphism by Lemma \ref{l:universal_enveloping_generated}.
\end{proof}
We consider multiindices $\alpha=(\alpha_1,\dots,\alpha_m)=\sum_{i=1}^m\alpha_ie_i\in\Zm$. The \emph{Poincar\'e series} of a $\Zm$-graded vector space $V=\bigoplus_{\alpha\in\Zm}V_\alpha$ is the formal power series
$$P(V;\lambda):=\sum_{\alpha\in\Zm}\dim(V_\alpha)\cdot\lambda^\alpha\in\ZZ[[\lambda_1,\dots,\lambda_m]],\quad \lambda^\alpha:=\prod_{i=1}^m\lambda_i^{\alpha_i}.$$
We identify a subset $J\subset[m]$ with the multiindex $\sum_{j\in J}e_j\in\Zm$, so $\lambda^J=\prod_{j\in J}\lambda_j$.

\begin{proposition}
\label{p:N_K_poincare_series}
    For $\alpha\in\Zm$, denote by $n_\alpha(\K):=\dim_{\ZZ_2}(N_\K)_\alpha$ the dimensions of $\Zm$-graded components of the Lie algebra $N_\K$. Then the numbers $n_\alpha(\K)$ are determined by the identity $$\prod_{\alpha\in\Zm}(1-\lambda^\alpha)^{n_\alpha(\K)}=\sum_{J\subset[m]}(1-\chi(\K_J))\lambda^{J}\in\ZZ[[\lambda_1,\dots,\lambda_m]].$$
    In particular, the dimensions of ordinary graded components, $n_i(\K):=\dim_{\ZZ_2}(N_\K)_i$, are determined by the identity
    $$\prod_{i\geq 1}(1-t^i)^{n_i(\K)}=\sum_{J\subset[m]}(1-\chi(\K_J))t^{|J|}\in\ZZ[[t]].$$
\end{proposition}
\begin{proof}
    Denote by $Sym(V)$ the symmetric algebra of a graded vector space $V$ (i.e. the polynomial algebra on the additive basis of $V$). The left hand side of the identity is equal to $1/P(Sym(N_\K);\lambda)$, while the right hand side is equal to $1/P(H_*(\Omega\ZK;\ZZ_2);\lambda)$ by \cite[Theorem 4.8]{vylegzhanin22} (in this paper, a different multigrading $\deg u_i=(-1,2e_i)\in\ZZ\times\Zm$ was considered). The algebras $H_*(\Omega\ZK;\ZZ_2)\cong U(N_\K)$ and $Sym(N_\K)$ are additively isomorphic by the Poincar\'e--Birkhoff--Witt theorem, so they have the same Poincar\'e series. The second identity is obtained by the substitution $\lambda_1=\dots=\lambda_m=t$.
\end{proof}
\begin{remark}
The same numbers $n_k(\K)$, denoted by $D_{k+1}$, appear in the description of homotopy type of moment-angle complexes in \cite[Theorem 1.2]{vylegzhanin}.
\end{remark}

\sloppy
A minimal presentation of the algebra $H_*(\Omega\ZK;\k)$ is known for flag $\K$ and any field $\k$ \cite[Theorem 1.1]{vylegzhanin}: we take GPTW generators as generators, while the relations are described in terms of simplicial 1-cycles generating first homology of full subcomplexes in $\K$. This presentation is minimal: any presentation of $H_*(\Omega\ZK;\k)$ has at least the same number of generators, and similarly with relations. We use this result to give a presentation of $N_\K$ by generators and relations, and give a criterion of freeness.
\begin{theorem}
\label{t:N_K_presentation}
Let $\K$ be a flag complex on $[m]$, and $\k=\ZZ_2$. Then the Lie algebra $N_\K$ is generated by the GPTW generators modulo the $\sum_{J\subset[m]}\dim_{\ZZ_2} H_1(\K_J;\ZZ_2)$ relations described in \cite[Theorem 1.1]{vylegzhanin}. This presentation is minimal.
\end{theorem}
\begin{proof}
By construction, the relations are Lie polynomials on the GPTW generators. Consider the Lie algebra $N$ generated by GPTW generators modulo these relations. By \cite[Corollaire 1.3.8]{lemaire}, the universal enveloping algebra $UN$ is presented by the same generators and relations, hence $UN\cong H_*(\Omega\ZK;\ZZ_2)\cong UN_\K$ by \cite[Theorem 1.1]{vylegzhanin} and Theorem \ref{t:U(N_K)}.

The same relations hold in the algebra $U(L_\K)\supset H_*(\Omega\ZK;\ZZ_2)$, so they hold in the Lie subalgebra $N_\K\subset U(L_\K)$ generated by the GPTW generators. We obtain a surjective map of Lie algebras $N\to N_\K$, hence a surjective map of associative algebras $UN\to UN_\K\cong H_*(\Omega\ZK;\ZZ_2)$. Since $UN\cong UN_\K$, this map is injective. So $N\cong N_\K$.
\end{proof}

\begin{theorem}
\label{t:N_K_is_free}
The following three conditions are equivalent:
\begin{enumerate}
    \item $N_\K$ is a free Lie algebra over $\ZZ_2$;
    \item $N_\K$ is freely generated by the GPTW generators, $N_\K\cong \FL_{\ZZ_2}(\GPTW)$;
    \item $\K^1$ is a chordal graph.
\end{enumerate}
\end{theorem}
\begin{proof}
    (2)$\Rightarrow$(1) is clear.
    
    If the graph $\K^1$ is not chordal, there is a full subcomplex $\K_J$ isomorphic to a $k$-cycle, $k\geq 4$. Then $H_1(\K_J;\ZZ_2)\cong\ZZ_2$, so in any presentation of the Lie algebra $N_\K$ there is at least one relation. So $N_\K$ is not free. This proves the implication (1)$\Rightarrow$(3).

Now let $\K^1$ be a chordal graph. The algebra $H_*(\Omega\ZK;\k)$ is a free algebra on the GPTW generators \cite[Theorem 4.6]{gptw}, so $U(N_\K)\cong T_{\ZZ_2}(\GPTW)$. Hence $N_\K\cong \FL_{\ZZ_2}(\GPTW)$. (An alternative argument: if $\K^1$ is chordal, then any full subcomplex $\K_J$ is homotopy equivalent to a disjoint union of points, hence $\sum_{J\subset[m]}\dim_{\ZZ_2} H_1(\K_J;\ZZ_2)=0$). This proves the implication (3)$\Rightarrow$(2).
\end{proof}
Note the similar result of Panov and Veryovkin \cite[Corollary 4.4]{pv}: the commutator subgroup $\RC'_\K=\gamma_2(\RC_\K)$ of the right-angled Coxeter group is free if and only if $\K^1$ is chordal graph, and in this case the analogues of GPTW generators  freely generate $\RC'_\K$.

\begin{conjecture}
    For any commutative ring $\k$ with unit, the Lie subalgebra $$N_{\K,\k}:=\langle\GPTW\rangle_{Lie}\subset L_{\K,\k}:=\FL_\k(\mu_1,\dots,\mu_m)/([\mu_i,\mu_j]=0,~\forall\{i,j\}\in\K)$$ is a free $\k$-module, and $(N_{\K,\k})_\alpha\simeq\k^{\oplus n_\alpha(\K)}$ for all $\alpha\in\Zm$. Moreover, analogues of Theorems \ref{t:N_K_presentation}, \ref{t:N_K_is_free} hold.
\end{conjecture}
Note that the proof of Theorem \ref{t:N_K_presentation} does not apply if $\mathrm{char}\,\k\neq 2$: we used that the commutators $ab-ba$ coincide with the graded commutators $ab-(-1)^{|a|\cdot|b|}ba$.

\subsection{Applications to the conjecture}
Recall that the numbers $n_i(\K):=\dim_{\ZZ_2} (N_\K)_i$ are computed in Proposition \ref{p:N_K_poincare_series}.
\label{subsec:relation with conj}
\begin{proposition}
\label{p:conjectural_poincare_series}
For a flag simplicial complex $\K$, the following statements are equivalent:
\begin{enumerate}
    \item Conjecture \ref{cnj:main_conjecture} holds for $\K$;
    \item $\dim_{\ZZ_2} L_k(\RC_\K)= n_2(\K)+\dots+n_k(\K)$ for all $k\geq 2$;
    \item $\dim_{\ZZ_2} L_k(\RC_\K)\geq n_2(\K)+\dots+n_k(\K)$ for all $k\geq 2$.
\end{enumerate}
\end{proposition}
\begin{proof}
By Theorem \ref{t:main_surjective_homomorphism},
\begin{equation}
\label{eq:dim-comparison}
\dim L_k(\RC_\K)\leq \dim(N_\K[t])_k=n_2(\K)+\dots+n_k(\K)
\end{equation}
for all $k\geq 2$,
so (2) and (3) are equivalent.

Clearly, $L'(\RC_\K)\cong N_\K[t]$ if and only if \eqref{eq:dim-comparison} turns into equality for all $k\geq 2$. It follows that (1) and (2) are equivalent.
\end{proof}
We show that Conjecture \ref{cnj:main_conjecture} holds for graded components of degree $2$ and $3$.
\begin{proposition}
\label{p:partial_confirmation}
For $i=2,3$ we have $\dim_{\ZZ_2} L_i(\RC_\K)=n_2(\K)+\dots+n_i(\K)$.
\end{proposition}
\begin{proof}
    By Proposition \ref{p:N_K_poincare_series},
    $$\prod_{i\geq 2}(1-x^i)^{n_i(\K)}=\sum_{J\subset[m]}(1-\chi(\K_J))x^{|J|}.$$

    For a simplicial complex $\L$ on $k$ vertices, let $c_\K(\L)$ be a number of $k$-element subsets $J\subset[m]$ such that the complexes $\K_J$ and $\L$ are isomorphic (possibly, after a permutation of vertices). By going over all flag simplicial complexes on $\leq 3$ vertices up to an isomorphism, we obtain: the right hand side of the identity is equal to
    $$1-c_\K(\pp)x^2-
    (2c_\K(\ppp)+c_\K(\pepp))x^3+O(x^4).$$
    The left hand side is equal to
    $$(1-x^2)^{n_2}(1-x^3)^{n_3}(1+O(x^4))=1-n_2x^2-n_3x^3+O(x^4);$$
    thus $n_2(\K)=c_\K(\pp)$ and $n_3(\K)=2c_\K(\ppp)+c_\K(\pepp)$.
On the other hand, \cite[Theorem 4.5]{veryovkin} describes the groups $L_i(\RC_\K)$ for $i=2,3$: the basis of these groups is given by the GPTW generators of degree $2,3$ and by the elements $[\og_i,[\og_i,\og_j]]$, where $\{i,j\}\notin\K$ and $i<j$. The GPTW generators of degree $2$ are in one-to-one correspondence with full subcomplexes of $\K$ isomorphic to $\pp$\,. In degree $3$, there are precisely two GPTW generators for each full subcomplex isomorphic to \,$\ppp$\,, and one GPTW generator for each full subcomplex isomorphic to \,$\pepp$\,. Thus, in our notation, 
$\dim_{\ZZ_2}L_2(\RC_\K)=c_\K(\pp)$ and $\dim_{\ZZ_2}L_3(\RC_\K)=c_\K(\pp)+2c_\K(\ppp)+c_\K(\pepp)$.
\end{proof}

Finally, we consider the special case when $\K^1$ is a chordal graph.
\begin{proposition}
    \label{p:conjecture_chordal_case}
    Let $\K^1$ be a chordal graph. If Conjecture \ref{cnj:main_conjecture} holds, then we have an isomorphism of graded Lie algebras
    $$\psi:\FL_{\ZZ_2}(\GPTW)[t]\overset\simeq\longrightarrow L'(\RC_\K),\quad \psi(xt^n):=h^n(\varphi'(x)).$$
\end{proposition}
\begin{proof}
By Theorem \ref{t:N_K_is_free}, in this case $N_\K\cong \FL_{\ZZ_2}(\GPTW)$.
\end{proof}
If $\K=\Delta^{m_1-1}\sqcup\dots\sqcup\Delta^{m_s-1}$ is a disjoint union of simplices (in other words, if $\RC_\K\cong \ZZ_2^{m_1}\ast\dots\ast\ZZ_2^{m_s}$), a similar result is stated in the thesis of R. Prener \cite[Theorem 4.39]{prener}. For example, ``G-simple basic commutators'' (GSBC) of \cite[Definition 3.03]{prener} are similar to our GPTW generators.  Prener's arguments use only the fact that $\RC_\K'$ is a free group, hence can be generalised to the chordal case. However, we were not able to reconstruct the details of the proof of the key lemma \cite[Lemma 4.37]{prener}.
\subsection{Examples and computations}
\label{subsec:examples}

By Theorem \ref{t:main_surjective_homomorphism}, we have an exact sequence of Lie algebras $$N_\K[t]\overset\psi\longrightarrow L(\RC_\K)\to CL_{\ZZ_2}(\og_1,\dots,\og_m)\to 0,$$
where $CL_\k(\dots)$ is the abelian Lie algebra over a field $\k$. If Conjecture \ref{cnj:main_conjecture} holds for $\K$, we get an extension $0\to N_\K[t]\to L(\RC_\K)\to CL_{\ZZ_2}(\og_1,\dots,\og_m)\to 0$.
Then the commutators in $L(\RC_\K)$ can be computed as follows:
\begin{enumerate}
    \item The commutator of $xt^k,yt^\ell\in N_\K[t]$ is equal to $[x,y]t^{k+\ell}\in N_\K[t]$, where $[x,y]$ is the commutator in $N_\K$;
    \item The commutator $[\og_i,\og_j]$ equals to a GPTW generator $[\mu_i,\mu_j]\in N_\K[t]$ if $\{i,j\}\notin\K$ and is zero if $\{i,j\}\in\K$;
    \item The commutator $[\og_i,y]\in N_\K[t]$ for $y\in\GPTW$ can be computed by applying Proposition \ref{p:removing repeats} to $a=[\og_i,y]$ and then applying Proposition \ref{p:gptw generate cwr} to the commutators $a'_i$.
    \item The commutator $[\og_i,y]\in N_\K[t]$ for arbitrary $y\in N_\K$ is computed inductively using (3), the Jacobi identity and properties of the operation $h$. (Here we use that $y$ is a Lie polynomial on GPTW generators.)
    \item The commutator $[\og_i,yt^\ell]\in N_\K[t]$ for $y\in N_\K$ equals $[\og_i,y]t^\ell$, where $[\og_i,y]$ is computed above.
\end{enumerate}
Even though a multiplication table for $L(\RC_\K)$ can be computed this way, it would be desirable to have a presentation of this Lie algebra by multiplicative generators and relations. Note that $\og_1,\dots,\og_m$ is a minimal set of generators.
\begin{problem}
Assuming Conjecture \ref{cnj:main_conjecture}, find a defining set of relations between the standard multiplicative generators $\og_1,\dots,\og_m$ of the Lie algebra $L(\RC_\K)$.
\end{problem}

\begin{example}
    Let $\K$ be a disjoint union of two points. Then the set $\GPTW=\{[\og_1,\og_2]\}$ is a singleton, so the free Lie algebra $\FL_{\ZZ_2}(\GPTW)\cong\ZZ_2\cdot [\og_1,\og_2]$ is abelian; it follows that the polynomial ring is also an abelian Lie algebra:
    $$\FL_{\ZZ_2}(GPTW)[t]\cong CL_{\ZZ_2}(\{xt^k:k\geq 0\})=\bigoplus_{k\geq 0}\ZZ_2 \cdot xt^k,~[xt^k,xt^\ell]=0.$$  
    By Theorem \ref{t:main_surjective_homomorphism} we have a surjective map of Lie algebras
    $$\psi:CL_{\ZZ_2}(\{xt^k:k\geq 0\})\twoheadrightarrow L'(\RC_\K),\quad xt^k\mapsto h^k([a_1,a_2])=[\underbrace{a_1,[a_1,\dots[a_1}_{k+1\text{ times}},a_2]\dots]].$$
    By \cite[Example 4.3]{veryovkin}, this map is bijective. Thus Conjecture \ref{cnj:main_conjecture} holds for $\K=\bullet\bullet$. To describe the Lie bracket in $L(\RC_\K)\simeq\ZZ_2\cdot\{\og_1,\og_2\}\oplus \bigoplus_{n\geq 0}\ZZ_2\cdot h^n([\og_1,\og_2])$, it is sufficient to compute the elements $[\og_1,[\og_1,\og_2]]$ and $[\og_2,[\og_1,\og_2]]$. Both elements are equal to $h([\og_1,\og_2])$ by part (4) of Theorem \ref{t:main_surjective_homomorphism}. See \cite[Proposition 4.4]{veryovkin} for a presentation of the Lie algebra $L(\RC_\K)$ by generators and relations.
\end{example}
\begin{example}
    Let $\K$ be a complex on three vertices with the single edge $\{1,3\}$. We have $[\og_1,\og_3]=0$ and $\GPTW=\{a:=[\og_2,\og_1],~b:=[\og_3,\og_2],c:=[\og_1,[\og_3,\og_2]]\}.$ We compute the commutators of these elements with the basic elements $\og_1,\og_2,\og_3$:
    \begin{gather*}[\og_1,a]=[\og_2,a]=h(a),~[\og_1,b]=c,~[\og_2,b]=[\og_3,b]=h(b),~[\og_1,c]=h(c);\\
    [\og_3,a]:=[\og_3,[\og_2,\og_1]]=[[\og_3,\og_2],\og_1]+[\og_2,[\og_3,\og_1]]=c+[\og_2,0]=c;\\
    [\og_2,c]:=[\og_2,[\og_1,[\og_3,\og_2]]]=[[\og_2,\og_1],[\og_3,\og_2]]+[\og_1,[\og_2,[\og_2,\og_3]]]\\
    =[a,b]+h([\og_1,[\og_2,\og_3]])=[a,b]+h(c);\\
    [\og_3,c]=[[\og_3,\og_1],[\og_3,\og_2]]+[\og_1,[\og_3,[\og_3,\og_2]]]=[0,[\og_3,\og_2]]+h(c)=h(c).
    \end{gather*}

We obtain an extension of Lie algebras
$$0\to \Big(\FL_{\ZZ_2}(a,b,c)[t]\Big)/\Ker\psi\to L(\RC_\K)\to CL_{\ZZ_2}(\og_1,\og_2,\og_3)\to 0,$$ where the commutators are computed as follows:
$[xt^k,yt^\ell]=[x,y]t^{k+\ell},$ $[\og_1,\og_2]=a$, $[\og_1,\og_3]=0,$ $[\og_2,\og_3]=b$,
    \begin{align*}
    [\og_1,a]&=at,
    &
    [\og_2,a]&=at,
    &
    [\og_3,a]&=c;\\
    [\og_1,b]&=c,
    &
    [\og_2,b]&=bt,
    &
    [\og_3,b]&=bt;\\
    [\og_1,c]&=ct,
    &
    [\og_2,c]&=ct+[a,b],
    &
    [\og_3,c]&=ct.    
    \end{align*}
If Conjecture \ref{cnj:main_conjecture} holds for $\K$, then $\Ker\psi=0$, so  the multiplicative structure of $L(\RC_\K)$ can be reconstructed from the identities above.
\end{example}
\begin{example}
    Let $\K$ be a boundary of pentagon. By Theorem \ref{t:U(N_K)} and the description of the algebra $H_*(\Omega\ZK;\k)$ \cite[Theorem 3.2]{veryovkin_pontryagin}, in this case the $\ZZ_2$-Lie algebra $N_\K$ is presented by the generators
    \begin{gather*}\alpha_1=[\og_3,\og_1],~\alpha_2=[\og_4,\og_1],~\alpha_3=[\og_4,\og_2],~\alpha_4=[\og_5,\og_2],~\alpha_5=[\og_5,\og_3],\\
    \beta_1=[\og_4,[\og_5,\og_2]],~\beta_2=[\og_3,[\og_5,\og_2]],~\beta_3=[\og_1,[\og_5,\og_3]],\\\beta_4=[\og_3,[\og_4,\og_1]],~\beta_5=[\og_2,[\og_4,\og_1]]\end{gather*}
    modulo the single relation $[\alpha_1,\beta_1]+\dots+[\alpha_5,\beta_5]=0.$
    By Proposition \ref{p:N_K_poincare_series}, the numbers $\dim (N_\K)_i=n_i(\K)$ satisfy $\prod_{i\geq 1}(1-t^i)^{n_i(\K)}=1-5t^2-5t^3+t^5$.
    An additive basis of $N_\K$ can be obtained e.g. using the Gr\"obner--Shirshov theory.
\end{example}

\end{document}